\documentclass[12pt]{amsart}
\usepackage{fancyhdr}
\usepackage{stmaryrd}
\usepackage{calrsfs}
\usepackage{graphicx}

\usepackage{amsmath}
\usepackage[nospace,noadjust]{cite}
\usepackage{amsfonts}
\usepackage{amssymb,enumerate}
\usepackage{amsthm}
\usepackage{cite}
\usepackage{comment}
\usepackage{color}
\usepackage[all]{xy}
\usepackage{hyperref}
\usepackage{lineno}

\newtheorem*{maintheorem*}{Main Theorem}
\newtheorem{theorem}{Theorem}[section]
\newtheorem{prop}[theorem]{Proposition}
\newtheorem{question}[theorem]{Question}

\newtheorem{lemma}[theorem]{Lemma}
\newtheorem{cor}[theorem]{Corollary}

\theoremstyle{definition}

\newtheorem{example}[theorem]{Example}
\numberwithin{equation}{section}

\newcommand{\nn}{\mathbb{N}}
\newcommand{\pp}{\mathbb{P}}
\newcommand{\qq}{\mathbb{Q}}

\newcommand{\zz}{\mathbb{Z}}

\newcommand{\ii}{\mathcal{A}}
\newcommand{\uu}{\mathcal{U}}

\providecommand\ldb{\llbracket}
\providecommand\rdb{\rrbracket}

\hypersetup{
	pdftitle={OHF},
	colorlinks=true,
	linkcolor=blue,
	citecolor=cyan,
	urlcolor=wine
}



\begin{document}
	
	\mbox{}
	\title{Betti graphs and atomization of Puiseux monoids}
	
	\author{Scott Chapman}
	\address{Department of Mathematics and Statistics\\Sam Houston State University\\Huntsville, TX 77341}
	\email{scott.chapman@shsu.edu}
	
	\author{Joshua Jang}
	\address{Oxford Academy\\Cypress, CA 90630}
	\email{joshdream01@gmail.com}

	\author{Jason Mao}
	\address{Morris Hills School\\Rockaway, NJ 07866}
	\email{jmao142857@gmail.com}

	\author{Skyler Mao}
	\address{Saratoga School\\Saratoga, CA 95070}
	\email{skylermao@gmail.com}

\date{\today}

\begin{abstract}
	 Let $M$ be a Puiseux monoid, that is, a monoid consisting of nonnegative rationals (under addition). A nonzero element of $M$ is called an atom if its only decomposition as a sum of two elements in $M$ is the trivial decomposition (i.e., one of the summands is $0$), while a nonzero element $b \in M$ is called atomic if it can be expressed as a sum of finitely many atoms allowing repetitions: this formal sum of atoms is called an (additive) factorization of $b$. The monoid $M$ is called atomic if every nonzero element of $M$ is atomic. In this paper, we study factorizations in atomic Puiseux monoids through the lens of their associated Betti graphs. The Betti graph of $b \in M$ is the graph whose vertices are the factorizations of $b$ with edges between factorizations that share at least one atom. Betti graphs have been useful in the literature to understand several factorization invariants in the more general class of atomic monoids. 
\end{abstract}
\medskip

\maketitle


\bigskip
\section{Introduction}
\label{sec:intro}

Let $M$ be an (additive) monoid that is cancellative and commutative. We say that a non-invertible element of $M$ is an atom if it cannot be written in $M$ as a sum of two non-invertible elements, and we say that $M$ is atomic if every non-invertible element of $M$ can be written as a sum of finitely many atoms (allowing repetitions). A formal sum of atoms which add up to $b \in M$ is called a factorization $b$, while the number of atoms in a factorization $z$ (counting repetitions) is called the length of $z$. Assume now that $M$ is an atomic monoid. If $b$ is a non-invertible element of $M$, then the Betti graph of $b$ is the graph whose elements are the factorizations of $b$ and whose set of edges consists of all pairs of factorizations having at least one atom in common. A non-invertible element of $M$ is called a Betti element if its Betti graph is disconnected. For a more general notion of a Betti element, namely, the \emph{syzygies} of an $\nn^k$-graded module, see~\cite{MS04}. Following~\cite{fG17}, we say that an additive submonoid of $\qq$ is a Puiseux monoid if it consists of nonnegative rationals. Factorizations in the setting of Puiseux monoids have been actively investigated in the past few years (see the recent papers~\cite{CGG21,GGT21} and the references therein). The primary purpose of this paper is to further understand factorizations in Puiseux monoids, now through the lens of Betti graphs.
\smallskip

Betti graphs are relevant in the theory of non-unique factorization because various of the most relevant factorization and length-factorization (global) invariants are either attained at Betti elements or can be computed using Betti elements. For instance, Chapman et al.~\cite{CGLPR06} proved that the catenary degree of every finitely generated reduced monoid is attained at a Betti element. In addition, Chapman et al.~\cite{CGLMS12} used Betti elements to describe the delta set of atomic monoids satisfying the bounded factorization property (the catenary degree and the delta set are two of the most relevant factorization invariants). Betti elements have been significantly studied during the last two decades. For instance, they have been studied by Garc\'ia-S\'anchez and Ojeda~\cite{GO10} in connection to uniquely-presented numerical semigroups. In addition, Garc\'ia-S\'anchez et al.~\cite{GOR13} characterized affine semigroups having exactly one Betti element, and for those semigroups they explicitly found various factorization invariants, including the catenary degree and the delta set. In the same direction, Chapman et al.~\cite{CCGS21} recently proved that every length-factorial monoid that is not a unique factorization monoid has a unique Betti element. Even more recently, the sets of Betti elements of additive monoids of the form $(\nn_0[\alpha],+)$ for certain positive algebraic numbers $\alpha$ have been explicitly computed by Ajran et al.~\cite{ABLST23}.
\smallskip

This paper is organized as follows. In Section~\ref{sec:background}, we discuss most of the terminology and non-standard results needed to follow the subsequent sections of content. In Section~\ref{sec:examples}, we provide some motivating examples and perform explicit computations of the sets of Betti elements of some Puiseux monoids. The discussed examples should provide certain intuition to better understand our main results. In Section~\ref{sec:atomization}, which is the section containing the main results of this paper, we discuss the notion of atomization, which is a method introduced by Gotti and Li in~\cite{GL23} that we can use to construct atomic Puiseux monoids with certain desired factorization properties. Indeed, most of the Puiseux monoids with applications in commutative ring theory can be constructed using atomization (see~\cite{aG74} and \cite{GL23}). As the main result of this paper, we describe the set of Betti elements of Puiseux monoids constructed by atomization, and we completely determine the sets of Betti elements for certain special types of atomized Puiseux monoids. Finally, we provide the following application of our main result: for any possible size $s$, there exists an atomic Puiseux monoid having precisely $s$ Betti elements.

\bigskip
\section{Background}
\label{sec:background}

\medskip
\subsection{General Notation and Terminology} 

Through this paper, we let $\nn$ denote the set of positive integers, and we set $\nn_0 := \nn \cup \{0\}$. In addition, we let $\pp$ stand for the set of primes. As it customary, we let $\zz$ and $\qq$ denote the set of integers and the set of rationals, respectively. If $b,c \in \zz$, then we let $\ldb b,c \rdb$ denote the discrete closed interval from $b$ to $c$; that is, $\ldb b,c \rdb := \{n \in \zz \mid b \le n \le c\}$ (observe that $\ldb b,c \rdb$ is empty if $b > c$). For a subset $X$ consisting of rationals and $q \in \qq$, we set
\[
	X_{\ge q} := \{x \in X \mid x \ge q\},
\]
and we define $X_{> q}$ in a similar manner. For $q \in \qq_{> 0}$, we let $\mathsf{n}(q)$ and $\mathsf{d}(q)$ denote the unique elements of $\nn$ satisfying that $\gcd(\mathsf{n}(q), \mathsf{d}(q)) = 1$ and $q = \mathsf{n}(q)/\mathsf{d}(q)$. For $p \in \pp$ and $n \in \nn$, the value $v_p(n)$ is the exponent of the largest power of $p$ dividing~$n$. Moreover, the $p$-\emph{adic valuation} is the map $v_p \colon \qq_{\ge 0} \to \zz$ defined by $v_p(q) = v_p(\mathsf{n}(q)) - v_p(\mathsf{d}(q))$ for $q \in \mathbb{Q}_{> 0}$ and $v_p(0) = \infty$. One can verify that the $p$-adic valuation satisfies the inequality $v_p(q_1 + \dots + q_n) \ge \min\{v_p(q_1), \dots, v_p(q_n) \}$ for every $n \in \nn$ and $q_1, \dots, q_n \in \qq_{> 0}$.

\medskip
\subsection{Monoids}

Throughout this paper, we tacitly assume that the term \emph{monoid} refers to a cancellative and commutative semigroup with an identity element. Unless we specify otherwise, monoids in this paper will be additively written. Let $M$ be a monoid. We let $M^\bullet$ denote the set $M \! \setminus \! \{0\}$. The group of invertible elements of $M$ is denoted by $\uu(M)$. Most of the monoids we consider in the scope of this paper have trivial groups of invertible elements.  A subset $M'$ of $M$ is called a \emph{submonoid} if $M'$ contains~$0$ and is closed under addition. A subset $S$ of $M$ is called a \emph{generating set} if the only submonoid of $M$ containing $S$ is $M$ itself, in which case we write $M = \langle S \rangle$. The monoid $M$ is called \emph{finitely generated} if it has a finite generating set; otherwise, $M$ is said to be \emph{non-finitely generated}. For $b,c \in M$, we say that $b$ \emph{divides} $c$ in $M$ and write $b \mid_M c$ if there exists $b' \in M$ such that $c = b + b'$. The monoid $M$ is called a \emph{valuation monoid} if for any pair of elements $b,c \in M$ either $b \mid_M c$ or $c \mid_M b$.

An non-invertible element $a \in M$ is called an \emph{atom} provided that for all $u,v \in M$ the fact that $a = u + v$ implies that $u \in \uu(M)$ or $v \in \uu(M)$. The set consisting of all the atoms of $M$ is denoted by $\mathcal{A}(M)$. Following Coykendall, Dobbs, and Mullins~\cite{CDM99}, we say that $M$ is antimatter if $\mathcal{A}(M)$ is empty. An element $b \in M$ is called \emph{atomic} if either $b$ is invertible or $b$ can be written as a sum of atoms (with repetitions allowed), while the whole monoid $M$ is called \emph{atomic} if every element of $M$ is atomic. A subset $I$ of $M$ is called an \emph{ideal} if
\[
	I + M := \{b + m \mid b \in I \text{ and } m \in M \} \subseteq I.
\]
In addition, an ideal $I$ of $M$ is said to be \emph{principal} if there exists an element $b \in M$ such that the following equality holds:
\[
	I = b + M := \{b + c \mid c \in M\}.
\]
A sequence of ideals $(I_n)_{n \ge 1}$ is called \emph{ascending} if $I_n \subseteq I_{n+1}$ for every $n \in \nn$ and is said \emph{to stabilize} if there exists $N \in \nn$ such that $I_n = I_N$ for every $n \in \nn$ with $n \ge N$. The monoid $M$ satisfies the \emph{ascending chain condition on principal ideals} (ACCP) if every ascending chain of principal ideals of $M$ stabilizes. It is well known that every monoid satisfying the ACCP is atomic \cite[Proposition~1.1.4]{GH06}. The converse does not hold, and we will discuss examples in the next sections.

\medskip
\subsection{Factorizations} 

Let $M$ be a monoid. The set $M_{\text{red}} := \{b + \uu(M) \mid b \in M\}$ is also a monoid under the natural addition induced by that of $M$ (one can verify that $M_{\text{red}}$ is atomic if and only if $M$ is atomic). We let $\mathsf{Z}(M)$ denote the free commutative monoid on the set $\mathcal{A}(M_{\text{red}})$, that is, the monoid consisting of all formal sums of atoms in $\mathcal{A}(M_{\text{red}})$. The monoid $\mathsf{Z}(M)$ plays an important role in this paper, and the formal sums in $\mathsf{Z}(M)$ are called \emph{factorizations}. The \emph{greatest common divisor} of two factorizations $z$ and $z'$ in $\mathsf{Z}(M)$, denoted by $\text{gcd}(z,z')$, is the factorization consisting of all the atoms $z$ and $z'$ have in common (counting repetitions). If a factorization $z \in \mathsf{Z}(M)$ consists of $\ell$ atoms of $M_{\text{red}}$ (counting repetitions), then we call $\ell$ the \emph{length} of $z$, in which case we often write $|z|$ as an alternative for $\ell$. We say that $a \in \ii(M)$ \emph{appears} in $z$ provided that $a + \uu(M)$ is one of the formal summands of $z$. 

There is a  unique monoid homomorphism $\pi_M \colon \mathsf{Z}(M) \to M_{\text{red}}$ satisfying $\pi(a) = a$ for all $a \in \ii(M_{\text{red}})$, which is called the \emph{factorization homomorphism} of $M$. 
When there seems to be no risk of ambiguity, we write $\pi$ instead of $\pi_M$. The set
\[
	\ker \pi := \{(z,z') \in \mathsf{Z}(M)^2 \mid \pi(z) = \pi(z') \}
\]
is called the \emph{kernel} of~$\pi$, and it is a congruence in the sense that it is an equivalence relation on $\mathsf{Z}(M)$ satisfying that if $(z,z') \in \ker \pi$, then $(z+w, z'+w) \in \ker \pi$ for all $w \in \mathsf{Z}(M)$. An element $(z,z') \in \ker \pi$ is called a \emph{factorization relation}. For each $x \in M$, we set
\[
	\mathsf{Z}(b) := \mathsf{Z}_M(b) := \pi^{-1}(b + \uu(M)) \subseteq \mathsf{Z}(M),
\]
and we call $\mathsf{Z}(b)$ the \emph{set of factorizations} of $b$. Observe that $\mathsf{Z}(u) = \{0\}$ if and only if $u \in \uu(M)$. If $|\mathsf{Z}(b)| = 1$ for every $b \in M$, then $M$ is called a \emph{unique factorization monoid} (UFM). For each $b \in M$, we set
\[
	\mathsf{L}(b) := \mathsf{L}_M(b) := \{|z| : z \in \mathsf{Z}(b)\} \subset \nn_0,
\]
and we call $\mathsf{L}(b)$ the \emph{set of lengths} of $b$. If $|\mathsf{L}(b)| = 1$ for every $b \in M$, then $M$ is called a \emph{half-factorial monoid} (HFM). Note that every UFM is an HFM (see \cite{sC14} for examples of HFMs that are not UFMs). Moreover, if $1 \le |\mathsf{L}(b)| < \infty$ for every $b \in M$, then $M$ is called a \emph{bounded factorization monoid} (BFM). It follows directly from the definitions that every HFM is a BFM. Cofinite submonoids of $(\nn_0,+)$ are called \emph{numerical monoids}, and every numerical monoid different from $\nn_0$ is a BFM that is not an HFM. In addition, it is well known that every BFM satisfies the ACCP \cite[Corollary~1]{fHK92}. The converse does not hold, and we will see examples illustrating this observation in coming sections. For a recent survey on factorizations on commutative monoids, see~\cite{GZ20}.

\subsection{Betti Elements and Betti Graphs} A finite sequence $z_0, \dots, z_k$ of factorizations in $\mathsf{Z}(M)$ is called a \emph{chain of factorizations} from $z_0$ to $z_k$ provided that $\pi(z_0) = \pi(z_1) = \dots = \pi(z_k)$ (here $\pi$ is the factorization homomorphism of $M$). Let $\mathcal{R}$ be the subset of $\mathsf{Z}(M) \times \mathsf{Z}(M)$ consisting of all pairs $(z,z')$ such that there exists a chain of factorizations $z_0, \dots, z_k$ from $z$ to $z'$ with $\text{gcd}(z_{i-1}, z_i) \neq 0$ for every $i \in \ldb 1,k \rdb$. It follows immediately that $\mathcal{R}$ is an equivalence relation on $\mathsf{Z}(M)$ that refines $\ker \pi$. Fix $b \in M$. We let $\mathcal{R}_b$ denote the set of equivalence classes of $\mathcal{R}$ inside $\mathsf{Z}(b)$, and the element $b$ is called a \emph{Betti element} provided that $|\mathcal{R}_b| \ge 2$.  The \emph{Betti graph} $\nabla_b$ of $b$ is the graph whose set of vertices is $\mathsf{Z}(b)$ having an edge between factorizations $z,z' \in \mathsf{Z}(x)$ precisely when $\gcd(z,z') \neq 0$. Observe that an element of $M$ is a Betti element if and only if its Betti graph is disconnected. We let $\text{Betti}(M)$ denote the set of Betti elements of $M$. 

\begin{example}
	Consider the numerical monoid $N := \langle 5,7,17,23 \rangle$. Using the SAGE package called \texttt{numericalsgps GAP}, we obtain that $|\text{Betti}(N)| = 3$: the Betti elements of $N$ are $28$, $30$, and $46$. Figure~\ref{fig:factorization graphs} shows the Betti graphs of both $40$ and $46$.
	\begin{figure}[h]
		\centering
		\includegraphics[width = 12cm]{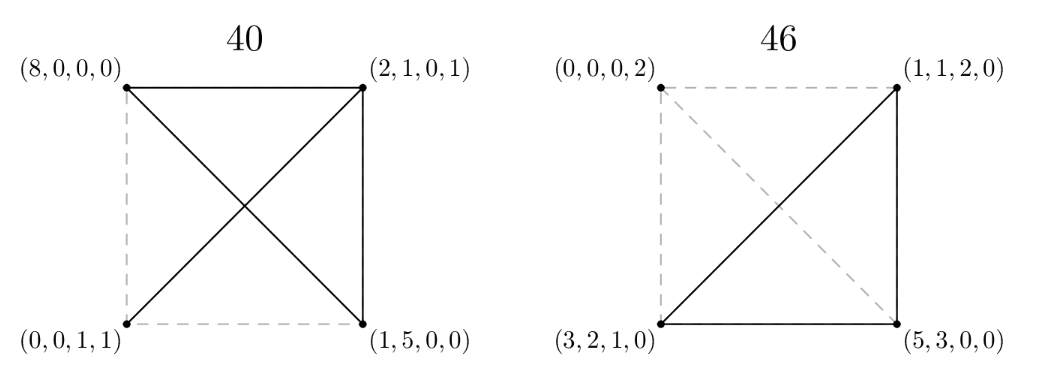}
		\caption{For $N := \langle 5,7,17,23 \rangle$, the figure shows the Betti graph of $40 \notin \text{Betti}(N)$ on the left and the Betti graph of $46 \in \text{Betti}(N)$ on the right.}
		\label{fig:factorization graphs}
	\end{figure}
\end{example}

It is clear that if a monoid is a UFM, then its set of Betti elements is empty. Following Coykendall and Zafrullah~\cite{CZ04} we say that a monoid $M$ is an \emph{unrestricted unique factorization monoid} (U-UFM) if every element of $M$ has at most one factorization. It follows directly from the definitions that every UFM is a U-UFM. We conclude this subsection characterizing U-UFMs in terms of the existence of Betti elements.

\begin{prop} \label{prop:U-UFM characterization}
	A monoid is a U-UFM if and only if its set of Betti elements is empty.
\end{prop}

\begin{proof}
	The direct implication follows immediately because if a monoid is a U-UFM then the Betti graph of each element has at most one vertex and is, therefore, connected.
	\smallskip
	
	For the reverse implication, assume that $M$ is a monoid containing no Betti elements. Now suppose, by way of contradiction, that $M$ is not a U-UFM. This means that there exists an element $x_0 \in M$ such that $|\mathsf{Z}(x_0)| \ge 2$. Let $z_0$ and $z'_0$ be two distinct factorizations of $x_0$. After dropping the common atoms of $z_0$ and $z'_0$ and subtracting the sum of such atoms from $x_0$, we can assume that $\gcd(z_0, z'_0)$ is the empty factorization. Since $x_0$ is not a Betti element, $z_0$ and $z'_0$ must be connected in $\nabla_{x_0}$, and so there exists a factorization $w_0$ of $x_0$ with $w_0 \neq z_0$ such that $\gcd(z_0, w_0)$ is nonempty. Now set $z_1 := z_0 - \gcd(z_0, w_0)$. Note that $z_1$ is a sub-factorization of $z_0$ satisfying  $|z_0| > |z_1|$ (because $\gcd(z_0, w_0)$ is nonempty). Take $x_1 \in M$ such that $z_1$ is a factorization of $x_1$, and observe that $x_1$ has at least two factorizations, namely, $z_1$ and $w_0 - \gcd(z_0, w_0)$. Because $x_1$ is not a Betti element, there must be a factorization $w_1$ of $x_1$ with $w_1 \neq z_1$ such that $\gcd(z_1, w_1)$ is nonempty. Now set $z_2 := z_1 - \gcd(z_1, w_1)$. Note that $z_2$ is a sub-factorization of $z_1$ satisfying  $|z_1| > |z_2|$ (because $\gcd(z_1, w_1)$ is nonempty). Proceeding in this fashion we can find a sequence $(z_n)_{n \ge 0}$ of factorizations in $M$ such that $|z_n| > |z_{n+1}|$ for every $n \in \nn_0$. However, this contradicts the well ordering principle. Hence $M$ must be a U-UFM.
\end{proof}

\bigskip
\section{Motivating Examples}
\label{sec:examples}

Our main purpose in this section is to discuss some examples that will serve as a motivation to establish our main results in the next section. The examples exhibited in this section will shed some light upon the potential sizes of the sets of Betti elements of Puiseux monoids.
\smallskip

It is clear and a special case of Proposition~\ref{prop:U-UFM characterization} that if a monoid is a UFM, then its set of Betti elements is empty. It is well known that a Puiseux monoid is a UFM if and only if it is an HFM. This occurs if and only if it can be generated by one element, in which case it is isomorphic to $\nn_0$ (see \cite[Proposition~4.3]{fG20}). There are, on the other hand, non-HFM atomic Puiseux monoids that contain finitely many Betti elements. The next two examples illustrate this observation.

\begin{example} \label{ex:Betti elements of f.g. PMs}
	Let $M$ be a finitely generated Puiseux monoid. If $M \neq \langle q \rangle$ for any element $q \in \qq_{>0}$, then it follows from \cite[Remark~2]{GO10} that $M$ contains at least a Betti element. Since $M$ is finitely generated, it must be isomorphic to a numerical monoid and, therefore,~$M$ has finitely many Betti elements (see \cite[Section~9.3]{GR09}). Thus, every finitely generated Puiseux monoid that is not generated by a single rational has a nonempty finite set of Betti elements.
\end{example}

It was proved in \cite[Proposition~3.5]{CCGS21} that if a monoid is an LFM that is not a UFM, then it contains exactly one Betti element, and it follows directly from~\cite[Proposition~5.7]{CCGS21} that a Puiseux monoid is an LFM if and only if it can be generated by two elements. However, there are non-finitely generated atomic Puiseux monoids with exactly one Betti element. This is illustrated in the following example.

\begin{example} \label{ex:Betti elements of the reciprocal PM}
	Consider the Puiseux monoid $M := \big\langle \frac1p \mid p \in \pp \big\rangle$. It is well known that~$M$ is atomic with $\mathcal{A}(M) = \big\{ \frac1p \mid p \in \pp \big\}$.  It follows from \cite[Example~3.3]{AG22} (see \cite[Proposition~4.2(2)]{fG22} for more details) that every element $q \in M$ can be written uniquely as
	\[
		q = c + \sum_{p \in \pp} c_p \frac1p,
	\] 
	where $c \in \nn_0$ and $c_p \in \ldb 0, p-1 \rdb$ for every $p \in \pp$ (here all but finitely many of the coefficients $c_p$'s are zero). From this, we can infer that for any element $q \in M$, the conditions $|\mathsf{Z}(q)| = 1$ and $1 \nmid_M q$ are equivalent. We claim that $\text{Betti}(M) = \{1\}$. To argue this equality, fix $q \in M^\bullet$. If $1 \nmid_M q$, then $|\mathsf{Z}(q)| = 1$ and so $\nabla_q$ is trivially connected, whence $q$ is not a Betti element. Assume, on the other hand, that $1 \mid_M q$ and, therefore, that $|\mathsf{Z}(q)| \ge 2$. Suppose first that $q \neq 1$. Because $M$ is atomic, we can write $q = 1 + \sum_{i=1}^k a_i$ for some $k \in \nn$ and $a_1, \dots, a_k \in \mathcal{A}(M)$. Observe that any two factorizations in $\mathsf{Z}(q)$ of the form $p \frac1p + a_1 + \dots + a_k$ with $p \in \pp$ are connected in the graph $\nabla_q$. In addition, any other factorization in $\mathsf{Z}(q)$ contains an atom $\frac{1}{p_0}$ for some $p_0 \in \pp$, so this factorization must be connected in $\nabla_q$ to the factorization $p_0 \frac1{p_0} + a_1 + \dots + a_k$. Hence $\nabla_q$ is connected when $1 \mid_ M q$ and $q \neq 1$, and so $q$ is not a Betti element in this case. Finally, we see that $q = 1$ is a Betti element: indeed, in this case, $\mathsf{Z}(1) = \big\{ p \frac1p \mid p \in \pp \}$, so the Betti graph of $1$ contains no edges. Hence $\text{Betti}(M) = \{1\}$.
\end{example}

The Puiseux monoids in the examples we have discussed so far have finitely many Betti elements. However, there exist atomic Puiseux monoids having infinitely many Betti elements. We provide an example showing this in the next section (Example~\ref{ex:Grams' monoid}). First, we need to discuss the notion of atomization.

\bigskip
\section{Atomization and Betti Elements}
\label{sec:atomization}

It turns out that we can construct Puiseux monoids with any prescribed number of Betti elements. Before doing so, we need to introduce the notion of atomization, which is a useful technique to construct Puiseux monoids satisfying certain desired properties. Let $(q_n)_{n \ge 1}$ be a sequence consisting of positive rationals, and let $(p_n)_{n \ge 1}$ be a sequence of pairwise distinct primes such that $\gcd(p_i, \mathsf{n}(q_i)) = \gcd(p_i, \mathsf{d}(q_j)) = 1$ for all $i,j \in \nn$. Following Gotti and Li~\cite{GL23}, we say that
\[
	M := \Big\langle \frac{q_n}{p_n} \ \Big{|} \ n \in \nn \Big\rangle
\]
is the \emph{Puiseux monoid} of $(q_n)_{n \ge 1}$ \emph{atomized} at $(p_n)_{n \ge 1}$. It is not hard to argue that $M$ is atomic with $\mathcal{A}(M) = \big\{ \frac{q_n}{p_n} \mid n \in \nn \big\}$ (see \cite[Proposition~3.1]{GL23} for the details). It turns out that we can determine the Betti elements of certain Puiseux monoids obtained by atomization. We will get into this matter in Theorem~\ref{thm:atomization and Betti elements}. First, we need the following technical lemma.

\begin{lemma} \label{lem:canonical decomposition of atomized PMs}
	Let $(q_n)_{n \ge 1}$ be a sequence consisting of positive rational numbers, and let $(p_n)_{n \ge 1}$ be a sequence of prime numbers whose terms are pairwise distinct such that $\gcd(p_i, \mathsf{n}(q_i)) = \gcd(p_i, \mathsf{d}(q_j)) = 1$ for all $i,j \in \nn$. Let $M$ be the Puiseux monoid of $(q_n)_{n \ge 1}$ atomized at $(p_n)_{n \ge 1}$. Then every element $q \in M$ can be uniquely written as follows:
	\begin{equation} \label{eq:canonical decomposition}
		q = n_q + \sum_{n \in \nn} c_n \frac{q_n}{p_n},
	\end{equation}
	where $n_q \in \langle q_n \mid n \in \nn \rangle$ and $c_n \in \ldb 0, p_n - 1 \rdb$ for every $n \in \nn$ (here $c_n = 0$ for all but finitely many $n \in \nn$).
\end{lemma}

\begin{proof}
	It suffices to prove the existence and uniqueness of the decomposition in~\eqref{eq:canonical decomposition} for every nonzero element $q \in M$. Fix $q \in M^\bullet$. Let~$N$ be the submonoid of $M$ generated by the sequence $(q_n)_{n \ge 1}$; that is,
	\[
		N := \langle q_n \mid n \in \nn \rangle.
	\]
	It follows from \cite[Proposition~3.1]{GL23} that $M$ is an atomic Puiseux monoid with
	\[
		\mathcal{A}(M) = \Big\{ \frac{q_n}{p_n} \ \Big{|} \ n \in \nn \Big\}.
	\]
	
 	For the existence of the decomposition in~\eqref{eq:canonical decomposition}, we first decompose $q$ as in~\eqref{eq:canonical decomposition} without imposing the condition that $c_n < p_n$ for all $n \in \nn$. Since $M$ is atomic, there is at least one way to decompose $q$ in the specified way (with $n_q = 0$). Among all such decompositions, choose $q = n_q + \sum_{n \in \nn} c_n \frac{q_n}{p_n}$ to be one minimizing the sum $\sum_{n \in \nn} c_n$. We claim that in the chosen decomposition, $c_n < p_n$ for every $n \in \nn$. Observe that if there existed $k \in \nn$ such that $c_k \ge p_k$, then
	\[
		q = n'_q + (c_k - p_k) \frac{q_k}{p_k} + \! \! \sum_{n \in \nn \setminus \{k\}} \! \! c_n \frac{q_n}{p_n},
	\]
	where $n'_q := n_q + q_k \in N$, would be another decomposition with smaller corresponding sum, which is not possible given the minimality of $\sum_{n \in \nn} c_n$. Hence every element $q \in M$ has a decomposition as in~\eqref{eq:canonical decomposition} satisfying $c_n \in \ldb 0, p_n - 1 \rdb$ for every $n \in \nn$. 
	
	For the uniqueness, suppose that $q$ has a decomposition as in~\eqref{eq:canonical decomposition} and also a decomposition $q = n'_q + \sum_{n \in \nn} c'_n \frac{q_n}{p_n}$ satisfying $n'_q \in N$ and $c'_n \in \ldb 0, p_n - 1 \rdb$ for every $n \in \nn$ (with $c'_n = 0$ for all but finitely many $n \in \nn$). Observe that, for each $n \in \nn$, the $p_n$-adic valuation of each element of $N$ is nonnegative and the $p_n$-adic valuation of $\frac{q_k}{p_k}$ is also nonnegative when $k \neq n$. Thus, for each $n \in \nn$, after applying $p_n$-adic valuation to both sides of $n'_q - n_q = \sum_{n \in \nn} (c_n - c'_n) \frac{q_n}{p_n}$, we find that $p_n \mid c_n - c'_n$, which implies that $c'_n = c_n$ (here we are using the fact that $c_n, c'_n \in \ldb 0, p_n - 1 \rdb$). Therefore $c'_n = c_n$ for every $n \in \nn$, and so $n'_q = n_q$. As a consequence, we can conclude that the decomposition in~\eqref{eq:canonical decomposition} is unique.
\end{proof}

With notation as in the statement of Lemma~\ref{lem:canonical decomposition of atomized PMs}, we call the equality in~\eqref{eq:canonical decomposition} the \emph{canonical decomposition} of $q$. We are now in a position to argue the main result of this section. Our proof of the following theorem is motivated by the argument given in Example~\ref{ex:Betti elements of the reciprocal PM}.

\begin{theorem} \label{thm:atomization and Betti elements}
	Let $(q_n)_{n \ge 1}$ be a sequence consisting of positive rational numbers, and let $(p_n)_{n \ge 1}$ be a sequence of prime numbers whose terms are pairwise distinct such that $\gcd(p_i, \mathsf{n}(q_i)) = \gcd(p_i, \mathsf{d}(q_j)) = 1$ for all $i,j \in \nn$. Let $M$ be the Puiseux monoid of $(q_n)_{n \ge 1}$ atomized at $(p_n)_{n \ge 1}$. Then the following statements hold.
	\begin{enumerate}
		\item For each $j \in \mathbb{N}$, the length-$p_j$ factorization $p_j \frac{q_j}{p_j}$ is an isolated vertex in $\nabla_{q_j}$.
		
		\item $\emph{Betti}(M) \subseteq \langle q_n \mid n \in \nn \rangle$.
		\smallskip
		
		\item $\{q_n \mid n \in \nn\} \subseteq \emph{Betti}(M)$ if $\langle q_n \mid n \in \nn \rangle$ is antimatter.
		\smallskip
		
		\item $\emph{Betti}(M) \subseteq \{q_n \mid n \in \nn\}$ if $\langle q_n \mid n \in \nn \rangle$ is a valuation monoid.
	\end{enumerate}
\end{theorem}

\begin{proof}
	Set $N := \langle q_n \mid n \in \nn \rangle$. As mentioned in Lemma~\ref{lem:canonical decomposition of atomized PMs}, the Puiseux monoid $M$ is atomic with
	\[
		\mathcal{A}(M) = \Big\{ \frac{q_n}{p_n} \ \Big{|} \ n \in \nn \Big\}.
	\]
	
	(1) Fix $j \in \nn$, and let us argue that $z:= p_j \frac{q_j}{p_j}$ is an isolated factorization in the Betti graph of $q_j$. If $|\mathsf{Z}(q_j)| = 1$, then we are done. Suppose, on the other hand, that $|\mathsf{Z}(q_j)| \ge 2$, and take $c_1, \dots, c_k \in \nn_0$ such that $z' := \sum_{i=1}^k c_i \frac{q_i}{p_i}$ is a factorization of $q_j$ in $M$ with $z \neq z'$ (we can assume, without loss of generality, that $k \ge j$). Because $v_{p_j}(q_j) = 0$, we can apply the $p_j$-adic valuation to both sides of the equality $q_j = \sum_{i=1}^k c_i \frac{q_i}{p_i}$ to find that $p_j \mid c_j$. Thus, the fact that $z \neq z'$ ensures that $c_j = 0$. As a consequence, $\gcd(z,z') = 0$. We can conclude, therefore, that $z$ is an isolated factorization in the Betti graph $\nabla_{q_j}$.
	\smallskip
	
	(2) Fix $q \in M$. It suffices to prove that if $q \notin N$, then $q$ is not a Betti element. To do so, assume that $q \notin N$. In light of Lemma~\ref{lem:canonical decomposition of atomized PMs}, we can write $q$ uniquely as
	\[
		q = n_q + \sum_{n \in \nn} c_n \frac{q_n}{p_n},
	\]
	where $n_q \in N$ and $c_n \in \ldb 0, p_n - 1 \rdb$ for every $n \in \nn$ (here $c_n = 0$ for all but finitely many $n \in \nn$). Since $q \notin N$, there exists $k \in \nn$ such that $c_k \neq 0$. In this case, the $p_k$-adic valuation of $q$ is negative and, therefore, every factorization of $q$ must contain the atom $\frac{q_k}{p_k}$, whence $\nabla_q$ is connected. Hence $\text{Betti}(M) \subseteq N$.
	\smallskip
	
	(3) Assume that $N$ is an antimatter monoid. For any $j \in \mathbb{N}$, recall from part~(1) that $z:=p_j\frac{q_j}{p_j}$ is an isolated factorization in the Betti graph $\nabla_{q_j}$. Also, since $N$ is an antimatter monoid, there exists $k \in \mathbb{N}$ and $s \in N^\bullet$ such that $q_j = q_k + s$. Now set
	\[
		z' := p_k \frac{q_k}{p_k} + z'',
	\]
	where $z''$ is a factorization of $s$ in $M$. Since $k \neq j$, we see that $z'$ is a factorization of $q_j$ in $M$ that is different from $z$. Since $z$ is isolated, $\gcd(z,z') = 0$, and so $\nabla_{q_j}$ is disconnected. Hence $q_j$ is a Betti element of $M$. As a result, the inclusion $\{q_n \mid n \in \mathbb{N}\} \subseteq \text{Betti}(M)$ holds.
	\smallskip
	
	(4) Lastly, assume that $N$ is a valuation monoid. Fix $q \in M^\bullet \setminus \{q_n \mid n \in \mathbb{N}\}$, and let us argue that $q$ is not a Betti element of $M$. If $q \notin N$, then it follows from part~(2) that $q \notin \text{Betti}(M)$. Hence we assume that $q \in N$. 
	Fix two factorizations
	\[
		z := \sum_{n \in \mathbb{N}} c_n\frac{q_n}{p_n} \quad \text{ and } \quad z' := \sum_{n \in \mathbb{N}} c_n'\frac{q_n}{p_n}
	\]
	of $q$ (here all but finitely many $c_n$ and all but finitely many $c_n'$ equal $0$). For each $n \in \nn$, the fact that $q \in N$ implies that $q$ has nonnegative $p_n$-adic valuation, and so after applying the $p_n$-adic valuation to both equalities $q = \sum_{n \in \mathbb{N}} c_n\frac{q_n}{p_n}$ and $q= \sum_{n \in \mathbb{N}} c'_n\frac{q_n}{p_n}$, we find that $p_n \mid c_n$ and $p_n \mid c'_n$. Because $q$ is nonzero, we can take $k,\ell \in \mathbb{N}$ such that $c_k \geq p_k$ and $c'_\ell \geq p_\ell$. Since $N$ is a valuation monoid, either $q_k \mid_N q_\ell$ or $q_\ell \mid_N q_k$. Assume, without loss of generality, that $q_\ell \mid_N q_k$. Then there exists $s \in N$ such that $q_k = q_\ell + s$. Now take a factorization $z_s$ of $s$ in $M$, and set
	\[
		z'' := z - p_k\frac{q_k}{p_k} + p_\ell\frac{q_\ell}{p_\ell} + z_s.
	\]
	Notice that $z''$ is a factorization of $q$ in $M$. As $q \not \in \{q_n \mid n \in \mathbb{N}\}$, it follows that $\gcd(z, z'') \neq 0$.  Also, the atom $\frac{q_\ell}{p_\ell}$ has nonzero coefficients in both $z'$ and $z''$, which implies that $\gcd(z', z'') \neq 0$. As the factorizations $z$ and $z'$ are both adjacent to $z''$ in the Betti graph $\nabla_q$, there is a length-$2$ path between them. Since $z$ and $z'$ were arbitrarily taken, the graph $\nabla_q$ is connected, which means that $q$ is not a Betti element. Hence $\text{Betti}(M) \subseteq \{q_n \mid n \in \mathbb{N}\}$.
\end{proof}

As an immediate consequence of Theorem~\ref{thm:atomization and Betti elements}, we obtain the following corollary.

\begin{cor} \label{cor:Betti elements of atomized PM}
	Let $(q_n)_{n \ge 1}$ be a sequence consisting of positive rational numbers, and let $(p_n)_{n \ge 1}$ be a sequence of prime numbers whose terms are pairwise distinct such that $\gcd(p_i, \mathsf{n}(q_i)) = \gcd(p_i, \mathsf{d}(q_j)) = 1$ for all $i,j \in \nn$. Let $M$ be the Puiseux monoid of $(q_n)_{n \ge 1}$ atomized at $(p_n)_{n \ge 1}$. If $\langle q_n \mid n \in \nn \rangle$ is an antimatter valuation monoid, then
	\[
		\emph{Betti}(M) = \{q_n \mid n \in \nn\}.
	\]
\end{cor}

As an application of Corollary~\ref{cor:Betti elements of atomized PM}, we can easily determine the set of Betti elements of the Grams' monoid.

\begin{example} \label{ex:Grams' monoid}
	Let $(p_n)_{n \ge 0}$ be the strictly increasing sequence whose underlying set consists of all odd primes, and consider the Puiseux monoid
	\[
		M := \Big\langle \frac{1}{2^n p_n} \ \Big{|} \ n \in \nn_0 \Big\rangle.
	\]
	The monoid $M$ is often referred to as the \emph{Grams' monoid} as it was the crucial ingredient in Grams' construction of the first atomic integral domain not satisfying the ACCP (see \cite{aG74} for the details of the construction). Observe that $M$ is the atomization of the sequence $\big( \frac{1}{2^n}\big)_{n \ge 0}$ at the sequence of primes $(p_n)_{n \ge 0}$. As a consequence, it follows from \cite[Proposition~3.1]{GL23} that $M$ is an atomic Puiseux monoid with
	\[
		\mathcal{A}(M) = \Big\{ \frac{1}{2^n p_n} \ \Big{|} \ n \in \nn_0 \Big\}.
	\]
	On the other hand, $M$ does not satisfy the ACCP because $\big( \frac{1}{2^n} + M \big)_{n \ge 0}$ is an ascending chain of principal ideals of $M$ that does not stabilize. Since $\big\langle \frac{1}{2^n} \mid n \in \nn_0 \big\rangle$ is an antimatter valuation monoid, it follows from Corollary~\ref{cor:Betti elements of atomized PM} that
	\[
		\text{Betti}(M) = \Big\{\frac{1}{2^n} \ \Big{|} \ n \in \nn_0 \Big\}.
	\]
	
\end{example}

As a final application of Theorem~\ref{thm:atomization and Betti elements}, we construct atomic Puiseux monoids with any prescribed number of Betti elements.

\begin{prop} \label{prop:sizes of Betti sets}
	For each $b \in \nn \cup \{\infty\}$, there exists an atomic Puiseux monoid $M$ such that $|\emph{Betti}(M)| = b$.
\end{prop}

\begin{proof}
	We have seen in Example~\ref{ex:Grams' monoid} that the Grams' monoid is an atomic Puiseux monoid, and we have also seen in the same example that the Grams' monoid has infinitely many Betti elements. Therefore it suffices to assume that $b \in \nn$.
	
	Fix $b \in \nn$. Now consider the sequence $(q_n)_{n \ge 1}$ whose terms are defined as $q_{kb + r} := r+1$ for every $k \in \nn_0$ and $r \in \ldb 0,b-1 \rdb$. Now let $(p_n)_{n \ge 1}$ be a strictly increasing sequence of primes such that $p_n > b$ for every $n \in \nn$. Then $\gcd(p_i, \mathsf{n}(q_i)) = \gcd(p_i, \mathsf{d}(q_j)) = 1$ for all $i,j \in \nn$. Let $M$ be the Puiseux monoid we obtain after atomizing the sequence $(q_n)_{\ge 1}$ at the sequence $(p_n)_{n \ge 1}$. It follows from \cite[Proposition~3.1]{GL23} that $M$ is an atomic Puiseux monoid with
	\[
		\mathcal{A}(M) := \Big\{ \frac{q_n}{p_n} \ \Big{|} \ n \in \nn \Big\}.
	\]
	Observe that $\langle q_n \mid n \in \nn \rangle = \langle 1, \dots, b \rangle = \nn_0$, which is a valuation monoid. As a consequence, it follows from part~(4) of Theorem~\ref{thm:atomization and Betti elements} that
	\[
		\text{Betti}(M) \subseteq \{q_n \mid n \in \nn\} = \ldb 1,b \rdb.
	\]
	Now fix $m \in \ldb 1,b \rdb$, and let us check that $m$ is a Betti element. To do this, first observe that the Betti graph $\nabla_m$ contains infinitely many vertices because
	\[
		\Big\{ p_{kb + (m-1)}\frac{m}{p_{kb + (m-1)}} \ \Big{|} \ k \in \nn \Big\} \subseteq \mathsf{Z}(m).
	\]
	Therefore $\nabla_m$ must be disconnected as it follows from part~(1) of Theorem~\ref{thm:atomization and Betti elements} that $p_{m-1} \frac{m}{p_{m-1}}$ is an isolated vertex in $\nabla_m$. Hence $\text{Betti}(M) = \ldb 1,b \rdb$, and so $|\text{Betti}(M)| = b$, as desired.
\end{proof}

Among the examples of atomic Puiseux monoids we have discussed so far, the only one having infinitely many Betti elements is the Grams' monoid, which does not satisfy the ACCP. However, there are Puiseux monoids containing infinitely many Betti elements that are FFMs. The following example illustrates this observation.

\begin{example} \label{ex:multiplicatively cyclic PM}
	Let $q$ be a non-integer positive rational, and consider the Puiseux monoid $M_q := \langle q^n \mid n \in \nn_0 \rangle$. It is well known that $M_q$ is atomic provided that $q^{-1} \notin \nn$, in which case, $\mathcal{A}(M_q) = \{q^n \mid n \in \nn_0\}$ (see \cite[Theorem~6.2]{GG18} and also \cite[Theorem~4.2]{CG22}). It follows from \cite[Lemma~4.3]{ABLST23} that 
	\[
		\text{Betti}(M_q) = \big\{ \mathsf{n}(q) q^n \mid n \in \nn_0 \big\}.
	\]
	Thus, $M_q$ is an atomic Puiseux monoid with infinitely many Betti elements. When $q>1$, it follows from~\cite[Theorem~5.6]{fG19} that $M_q$ is an FFM (in particular $M_q$ satisfies the ACCP).
\end{example}

As we have mentioned in Example~\ref{ex:Betti elements of f.g. PMs}, every finitely generated Puiseux monoid has finitely many Betti elements. Although the class of finitely generated monoids sits inside the class of FFMs (see \cite[Proposition~2.7.8]{GH06}), we have seen in Example~\ref{ex:multiplicatively cyclic PM} that inside the class of Puiseux monoids, the finite factorization property is not enough to guarantee that the set of Betti elements is finite.

On the other hand, every atomic Puiseux monoid with finitely many Betti elements we have discussed so far satisfies the ACCP: these include the Puiseux monoids discussed in Examples~\ref{ex:Betti elements of f.g. PMs} and~\ref{ex:Betti elements of the reciprocal PM} as well as the Puiseux monoids constructed in the proof of Proposition~\ref{prop:sizes of Betti sets}, which satisfy the ACCP in light of \cite[Theorem~4.5]{AGH21}. We have not been able to construct an atomic Puiseux monoid with finitely many Betti elements that does not satisfy the ACCP. Thus, we conclude this paper with the following question.

\begin{question}
	Does every atomic Puiseux monoid with finitely many Betti elements satisfy the ACCP?
\end{question}

\bigskip
\section*{Acknowledgments}

We thank our MIT PRIMES mentors Prof.\ Scott Chapman and Dr.\ Felix Gotti for introducing us to this area of research and for providing this research project  and valuable guidance during its completion. We also thank Dr.\ Harold Polo for helpful comments and feedback. Finally, we thank the MIT PRIMES program for creating this research opportunity in the first place.

\bigskip

\end{document}